\documentclass{amsart}
\usepackage{amsmath,mathrsfs}
\usepackage[all]{xy}

\newtheorem{thm}{Theorem}[section]

\newtheorem{prop}[thm]{Proposition}
\theoremstyle{definition}
\newtheorem{defn}[thm]{Definition}
\newtheorem{quest}[thm]{Question}
\theoremstyle{remark}

\numberwithin{equation}{thm}

\newcommand{\N}{\mathbb{N}}

\newcommand{\ran}{\mathrm{ran}}
\newcommand{\nleq}{\not \leq}
\newcommand{\ngeq}{\not \geq}

\newcommand{\RCA}{\mathsf{RCA}_0}
\newcommand{\ACA}{\mathsf{ACA}_0}
\newcommand{\WKL}{\mathsf{WKL}_0}

\newsymbol\upharpoonright 1316
\newsymbol\varnothing 203F
\newsymbol\nvdash 2330 
\newsymbol\nVdash 2331
\newsymbol\Vdash 130D
\newsymbol\nprec 2306
\newsymbol\npreceq 230E
\newsymbol\nleftrightarrow 233D
\newsymbol\nsubseteq 232A
\newsymbol\nsupseteq 232B

\begin{document}

\title{Infinite saturated orders}
\author{Damir D. Dzhafarov}
\address{University of Chicago}
\email{damir@math.uchicago.edu}
\thanks{The author was partially supported by an NSF Graduate Research Fellowship.  He is grateful to R. Suck for introducing him to this problem, and to E. Dzhafarov and his thesis advisers, R. Soare, D. Hirschfeldt, and A. Montalb\'{a}n, for helpful comments.}

\begin{abstract}
We generalize the notion of saturated order to infinite partial orders and give both a set-theoretic and an algebraic characterization of such orders.  We then study the proof theoretic strength of the equivalence of these characterizations in the context of reverse mathematics, showing that depending on one's choice of definitions it is either provable in $\RCA$ or equivalent to $\ACA$.
\end{abstract}

\maketitle

\section{Introduction}

Saturated orders were introduced by Suck in \cite{Su1} as a generalization of interval orders.  The latter, developed by Fishburn (see \cite{Fi1}), have been used extensively in the theory of measurement, utility theory, and various areas of psychophysics and mathematical psychology (see \cite{Fi1}, Chapter 2, for examples).  Suck applied the concept of saturated orders to the theory of knowledge spaces, as introduced by Doignon and Falmagne (see \cite{DoFa}), but he formulated it for finite orders only.  Since the study of knowledge spaces in general need not be restricted to finite structures, it is natural to ask whether the concept of saturation can be formulated for arbitrary partial orders.

In this note, we give such a formulation and show it to be equivalent to a certain algebraic characterization of partial orders.  We then look at the proof theoretic strength of this equivalence using the framework of reverse mathematics.  This answers questions of Suck raised at the {\em Reverse Mathematics: Foundations and Applications} workshop at the University of Chicago in November 2009.  Beyond an interest in the underlying combinatorial principles, the motivation for this kind of analysis comes from seeking a possible new basis by which to judge and compare competing quantitative approaches to problems in cognitive science.  The exploration of this interaction was one of the goals of the Chicago workshop.

\begin{defn}
Let ${\bf P} = (P,\leq_P)$ be a partial order.
\begin{enumerate}
\item An {\em interval representation} of $\bf P$ is a map $f$ from $P$ into the set of finite open intervals of some linear order ${\bf L} = (L,\leq_L)$ such that for all $p,p' \in P$, $p <_P p'$ if and only if $\ell <_L \ell'$ for all $\ell \in f(p)$ and $\ell' \in f(p')$;
\item $\bf P$ is an {\em interval order} if it admits an interval representation.
\end{enumerate}
\end{defn}

\begin{defn}[\cite{Su1}, Definitions 1 and 3] \label{D:set_finite}
Let ${\bf P} = (P, \leq_P)$ be a finite partial order.
\begin{enumerate}
\item A {\em set representation} of $\bf P$ is an injective map $\varphi : P \to \mathcal{P}(Q)$ for some set $Q$ such that $p <_P p'$ if and only if $\varphi(p) \subset \varphi(p')$ for all $p,p' \in P$.
\item A set representation $\varphi$ of $\bf P$ is {\em parsimonious} if $\left| \varphi(p) \right|  = \left| \bigcup_{p' <_P p} \varphi(p') \right| + 1$
for all $p \in P$.
\item $\bf P$ is {\em saturated} if $\left| P \right | = \left| \bigcup_{p \in P} \varphi(p) \right|$
for all parsimonious set representations $\varphi$ of $\bf P$.
\end{enumerate}
\end{defn}

Every finite partial order ${\bf P} = (P, \leq_P)$ admits at least one parsimonious set representation, namely $\pi : P \to \mathcal{P}(P)$ where $\pi(p) = \{p' \in P: p' \leq_P p\}$ for all $p \in P$.  Suck \cite[Definition 2]{Su1} calls this the {\em principal ideal representation} of $\bf P$.  The notion of saturation arose as a means of characterizing orders for which this is essentially the only parsimonious set representation (\cite{Su1}, p. 375).  Indeed, suppose $\varphi : P \to \mathcal{P}(Q)$ is parsimonious, and let $\alpha_\varphi: P \to Q$ be defined by setting $\alpha_\varphi(p)$ for each $p \in P$ to be the single element of $\varphi(p) - \bigcup_{p' <_P p} \varphi(p')$.  If $\bf P$ is saturated then $\alpha_\varphi$ must be a bijection between $P$ and $\bigcup_{p \in P} \varphi(p)$.  Let $\leq_Q$ be an ordering of the latter set defined by setting $q \leq_Q q'$ for each $q,q' \in \bigcup_{p \in P} \varphi(p)$ if and only if $q = \alpha_\varphi(p)$ and $q' = \alpha_\varphi(p')$ for some $p,p' \in P$ with $p \leq_P p'$.  Then $\alpha_\varphi$ is an isomorphism of $\bf P$ with $(\bigcup_{p \in P} \varphi(p), \leq_Q)$, and $\varphi(p) = \alpha_\varphi(\pi(p))$ for all $p \in P$.  Thus, up to a renaming of elements, $\varphi$ and $\pi$ are the same representation.

Suck \cite[Theorem 2]{Su1} showed that every finite interval order is a saturated order.  On the other hand, it is easy to build a saturated order which admits a suborder of type $\bf 2 \oplus 2$, i.e., a suborder isomorphic to $(\{a,b,c,d\},\leq)$ where $a \leq b$, $c \leq d$, $a \nleq d$ and $c \nleq b$ (see \cite{Su1}, Figure 2).  Such an order cannot be an interval order:

\begin{thm}[Fishburn \cite{Fi2}, p. 147; Mirkin \cite{Mi}]\label{T:intchar}
A partial order is an interval order if and only if it does not contain a suborder of type $\bf 2 \oplus 2$.
\end{thm}

If one recasts the condition of not containing a suborder of type $\bf 2 \oplus 2$ as
$$
(\forall p_0,p_1,p_2,p_3 \in P)[p_0 \not<_P p_1 \vee p_2 \not<_P p_3 \vee  p_0 \leq_P p_3 \vee p_2 \leq_P p_1],
$$
then the following definition and theorem provide a similar algebraic characterization of saturation.

\begin{defn}[\cite{Su2}, Definitions 5 and 6]\label{D:fan}
Let ${\bf P} = (P, \leq_P)$ be a finite partial order.
\begin{enumerate}
\item A {\em fan} in $\bf P$ is a subset $F$ of $P$ with at least two elements such that $\max F$ exists under $\leq_P$ and such that no elements of $F - \{\max F\}$ are pairwise $\leq_P$-comparable.
\item Two fans $F_0$ and $F_1$ in $\bf P$ are {\em parallel} if no element of $F_0$ is $\leq_P$-comparable with any element of $F_1$.
\item Two parallel fans $F_0$ and $F_1$ in $\bf P$ are {\em skewly topped} if there exists some $m \in P$ and some $i \in \{0,1\}$ such that
\begin{enumerate}
\item $m \geq_P \max F_i$,
\item $m \ngeq_P \max F_{1-i}$,
\item and $m \geq_P p$ for all $p \in F_{1-i} - \{\max F_{1-i}\}$.
\end{enumerate}
\end{enumerate}
\end{defn}

\begin{thm}[Suck \cite{Su2}, Theorem 5]\label{T:Suck}
A finite partial order is saturated if and only if every two parallel fans in it are skewly topped.
\end{thm}

We can now state the questions of Suck mentioned above.

\begin{quest}[Suck]\label{Q:Suck}
\
\begin{enumerate}
\item Does (some suitable analog of) Theorem \ref{T:Suck} hold for infinite partial orders?
\item If so, what are the set theoretic axioms necessary to carry out its proof? 
\end{enumerate}
\end{quest}

\noindent The second part is inspired by the work of Marcone \cite{Ma}, who investigated the reverse mathematical content of Theorem \ref{T:intchar}.  We refer the reader to Section \ref{S:RM} for a brief introduction to reverse mathematics, and Simpson \cite{Si} for a complete reference.  In the next section we give an affirmative answer to part (1) of Question \ref{Q:Suck}, and in Section \ref{S:RM} we consider possible answers to part (2).

\section{Infinite saturated orders}

In this section we formulate the concept of saturation for infinite partial orders and prove an analog of Theorem \ref{T:Suck}.  To begin, notice that set representations can be defined for infinite orders just as for finite ones.  The other parts of Definition \ref{D:set_finite}, however, need to be appropriately adjusted to the infinite setting.

\begin{defn}\label{D:set_infinite}
Let ${\bf P} = (P, \leq_P)$ be a partial order.
\begin{enumerate}
\item A set representation $\varphi : P \to \mathcal{P}(Q)$ of $\bf P$ is {\em parsimonious} if for all $p \in P$
\begin{enumerate}
\item $\left| \varphi(p)  - \bigcup_{p' <_P p} \varphi(p') \right| = 1$,
\item and for all $q \in \varphi(p)$, $\{q\} = \varphi(p') - \bigcup_{p'' < p'} \varphi(p'')$ for some $p' \leq_P p$.
\end{enumerate}
\item Given a parsimonious set representation $\varphi$ of $\bf P$, define $\alpha_\varphi : P \to Q$ by $\alpha_\varphi(p) = q$ for $p \in P$ if and only if $\{q\} = \varphi(p) - \bigcup_{p' <_P p} \varphi(p')$.
\item $\bf P$ is {\em saturated} if and only if $\alpha_\varphi$ is injective for all parsimonious set representations $\varphi$ of $\bf P$.
\end{enumerate}
\end{defn}

It is not difficult to check that for finite partial orders the new definitions agree with the old:

\begin{prop}
Let ${\bf P} = (P, \leq_P)$ be a finite partial order.
\begin{enumerate}
\item A set representation of $\bf P$ is parsimonious according to Definition \ref{D:set_finite} if and only if it is parsimonious according to Definition \ref{D:set_infinite}.
\item $\bf P$ is saturated according to Definition \ref{D:set_finite} if and only if it is saturated according to Definition \ref{D:set_infinite}.
\end{enumerate}
\end{prop}

%\begin{proof}
%For (1), let $\varphi: P \to \mathcal{P}(Q)$ be a set representation of $\bf P$ that is parsimonious according to Definition \ref{D:set_finite}.  We claim that for every $p \in P$ and every $q \in \varphi(p)$, there exists $p' \leq_P p$ such that $\{q\} = \varphi(p') - \bigcup_{p'' <_P p'} \varphi(p'')$.  So fix $p \in P$ and $q \in \varphi(p)$, and let $n \geq 0$ be the number of $\leq_P$-predecessors of $p$ in $P$.  Assume by induction that our claim holds for all elements with strictly less than $n$ many $\leq_P$-predecessors.  If $q \notin \bigcup_{p' <_P p} \varphi(p')$, then parsimony according to Definition \ref{D:set_finite} implies that $\{q \} = \varphi(p) - \bigcup_{p' < p} \varphi(p')$, as desired.  Otherwise, fix $p_0 <_P p$ such that $q \in \varphi(p_0)$.  Since $p_0$ has strictly less than $n$ many $\leq_P$-predecessors, it follows that $\{q \}  = \varphi(p') - \bigcup_{p'' < p'} \varphi(p'')$ for some $p' \leq_P p_0$.  Since $p' \leq_P p$, this proves our claim.

%To prove (2), note that if $\varphi : P \to \mathcal{P}(Q)$ is any set representation of $\bf P$, then $\alpha_\varphi : P \to \bigcup_{p \in P} \varphi(p)$ is a surjection (a fact which holds for all partial orders, finite and infinite; see Lemma 1 of \cite{Su2}).  Hence if $\bf P$ is finite, $|P| = \left|\bigcup_{p \in P} \varphi(p) \right|$ if and only if $\alpha_\varphi$ is injective.
%\end{proof}

\noindent In particular, the argument given following Definition \ref{D:set_finite} holds verbatim for infinite partial orders as long as parsimony and saturation are understood according to Definition \ref{D:set_infinite}.  Thus infinite saturated orders admit only one parsimonious set representation, and so the preceding definition does indeed capture the ``spirit'' of the concept.

We next generalize the notion of fan from Definition \ref{D:fan}; we shall see at the end of the section why fans alone would not suffice.

\begin{defn}\label{D:bouquet}
A {\em bouquet} in $\bf P$ is a subset $B$ of $P$ with at least two elements such that $\max B$ exists under $\leq_P$.
\end{defn}

%\begin{defn}
%Let ${\bf P} = (P, \leq_P)$ be a partial order.
%\begin{enumerate}
%\item A {\em bouquet} in $\bf P$ is a subset $B$ of $P$ with at least two elements such that $\max B$ exists.
%\item Two fans $B_0$ and $B_1$ in $\bf P$ are {\em parallel} if no element of $B_0$ is $\leq_P$-comparable with any element of $B_1$.
%\item Two parallel fans $B_0$ and $B_1$ in $\bf P$ are {\em skewly topped} if there exists some $m \in P$ and some $i \in \{0,1\}$ such that
%\begin{enumerate}
%\item $m >_P \max B_i$,
%\item $m \ngeq_P \max B_{1-i}$,
%\item and $m >_P p$ for all $p \in B_{1-i}$.
%\end{enumerate}
%\end{enumerate}
%\end{defn}

\noindent We define what it means for two bouquets to be {\em parallel} and {\em skewly topped} just as for fans.  If ${\bf P}$ is finite, or even just a partial order in which every element has only finitely many $\leq_P$-successors, then every two parallel bouquets $B_0$ and $B_1$ can be replaced by parallel fans $F_0$ and $F_1$ with the same respective maxima.  Namely, let $F_i = \{ b \in B_i : (\forall b' \geq_P b )[b' \in B_i \implies b' = \max B_i]\}$ for each $i$.  Then an element of $P$ skewly tops $B_0$ and $B_1$ if and only if it skewly tops $F_0$ and $F_1$, and conversely.  Thus we have:
\begin{prop}\label{P:fans_bouquets}
If ${\bf P} = (P,\leq_P)$ is a finite partial order then every two parallel fans in $\bf P$ are skewly topped if and only if every two parallel bouquets in $\bf P$ are skewly topped.
\end{prop}

%\begin{proof}
%Since every fan is a bouquet, we have only to prove the right-to-left direction.  So assume that every two parallel fans in $\bf P$ are skewly topped, and let $B_0$ and $B_1$ be a pair of parallel bouquets.  For each $i \in \{0,1\}$, let
%$$
%F_i = \{ b \in B_i : (\forall b' \geq_P b )[b' \in B_i \implies b' = \max B_i]\},
%$$
%which is well-defined since every element of $P$ has only finitely many $\leq_P$-successors.  Then $F_0$ and $F_1$ are parallel fans with respective maxima $\max B_0$ and $\max B_1$.  Let $m \in P$ skewly top the two, say with $m >_P \max F_i = \max B_i$.  Then $m \ngeq_P \max F_{1-i} = \max B_{1-i}$ but $b <_P m$ for every $b \in B_{1-i} - \{\max B_{1-i}\}$ because for every such $b$ there exists $b' \in F_{1-i}$ with $b \leq_P b' <_P m$.  Hence, $m$ skewly tops $B_0$ and $B_1$.
%\end{proof}

The following is the analog for infinite partial orders of Theorem \ref{T:Suck}.  Along with the preceding two propositions it also gives an alternative proof of that theorem, Suck's original one having been by induction on the size of the partial order.

\begin{thm}\label{T:main}
A partial order is saturated if and only if every two parallel bouquets in it are skewly topped.
\end{thm}

\begin{proof}
%[Proof of Theorem \ref{T:main}]

 ($\Longrightarrow $) Suppose $B_0$ and $B_1$ are two parallel bouquets in $\bf P$ that are not skewly topped.  Let $q^*$ be a symbol not in $P$, and let $Q = P \cup \{q^*\} - \{\max B_0,\max B_1\}$.  Let $\pi$ be the principal ideal representation of $\bf P$, and define $\varphi : P \to \mathcal{P}(Q)$ as follows.  If $p \geq_P \max B_0$ or $p \geq_P \max B_1$ let $\varphi(p) = \pi(p) \cup \{q^*\} - \{\max B_0, \max B_1\}$, and otherwise let $\varphi(p) = \pi(p)$.   We claim, first of all, that $\varphi$ is a set representation.  So fix distinct $p_0,p_1 \in P$ and note that if $p_0 \ngeq_P \max B_0, \max B_1$ and $p_1 \ngeq_P \max B_0, \max B_1$ then $\varphi(p_i) = \pi(p_i)$ for each $i$, meaning $\varphi(p_0) \neq \varphi(p_1)$ and $p_i <_P p_{1-i}$ if and only if $\varphi(p_i) \subset \varphi(p_{1-i})$.  This leaves the following cases to consider.

\medskip
\noindent {\em Case 1: for some $i, j \in \{0,1\}$,}
\begin{itemize}
\item $p_i \geq_P \max B_j$,
\item $p_{1-i} \ngeq_P \max B_0, \max B_1$.
\end{itemize}
Clearly $\varphi(p_0) \neq \varphi(p_1)$ since $q^* \in \varphi(p_i)$ and $q^* \notin \varphi(p_{1-i})$.  If $p_0$ and $p_1$ are $\leq_P$-comparable, it must be that $p_{1-i} <_P p_{i}$, so $\varphi(p_{1-i}) = \pi(p_{1-i}) \subset \pi(p_i)$.  And since $p_{1-i} \ngeq \max B_0, \max B_1$ we have $\max B_0,\max B_1 \notin \pi(p_{1-i})$, implying that $\varphi(p_{1-i})  \subseteq \pi(p_i) - \{\max B_0, \max B_1\} \subset \varphi(p_i)$.  Conversely, if $\varphi(p_0)$ and $\varphi(p_1)$ are comparable under inclusion, it must be that $\varphi(p_{1-i}) \subset \varphi(p_i)$.  Thus $\pi(p_{1-i}) \subseteq \varphi(p_i) - \{q^*\} \subseteq \pi(p_i)$. However, it cannot be that $\pi(p_{1-i}) = \pi(p_i)$ since this would mean that $\max B_j \leq_P p_{1-i}$, so we must have $\pi(p_{1-i}) \subset \pi(p_i)$ and hence $p_{1-i} <_P p_i$.

\medskip
\noindent {\em Case 2: for some $i,j \in \{0,1\}$,}
\begin{itemize}
\item $p_i \geq_P \max B_j$,
\item $p_i \ngeq_P \max B_{1-j}$,
\item $p_{1-i} \geq_P \max B_{1-j}$,
\item $p_{1-i} \ngeq_P \max B_{j}$.
\end{itemize}
 In this case we clearly cannot have $p_{1-i} <_P p_i$ or $p_i <_P p_{1-i}$.  We show that neither $\varphi(p_{1-i}) \subseteq \varphi(p_i)$ nor $\varphi(p_i) \subseteq \varphi(p_{1-i})$ can obtain.  Indeed, suppose it were the case that $\varphi(p_{1-i}) \subseteq \varphi(p_i)$ (the other case being symmetric).  Then every $p \in B_{1-j} - \{\max B_{1-j}\}$, being an element of $\pi(p_{1-i})$, would belong to $\varphi(p_i)$ and, not being $q^*$, also to $\pi(p_i)$.  Thus, we would have $p \leq_P p_i$, so $p_i$ would skewly top $B_0$ and $B_1$, a contradiction.

% Thus, if $p_i >_P \max B_j$ or $p_{1-i} >_P \max B_{1-j}$ then $p_i \in \varphi(p_i) - \varphi(p_{1-i})$ or $p_{1-i} \in \varphi(p_{1-i}) - \varphi(p_i)$, meaning $\varphi(p_{i}) \neq \varphi(p_{1-i})$.  And if $p_i = \max B_j$ and $p_{1-i} = \max B_{1-j}$, then $\varphi(p_i) \neq \varphi(p_{1-i})$ because the elements of $B_j - \{\max B_j\}$ belong to $\varphi(p_i)$ but not to $\varphi(p_{1-i})$.

%It remains to check that neither $\varphi(p_{1-i}) \subset \varphi(p_i)$ nor $\varphi(p_i) \subset \varphi(p_{1-i})$ can obtain.  Indeed, suppose it were the case that $\varphi(p_{1-i}) \subset \varphi(p_i)$.  Then every $p \in B_{1-j} - \{\max B_{1-j}\}$, being an element of $\pi(p_{1-i})$, would belong to $\varphi(p_i)$ and, not being $q^*$, also to $\pi(p_i)$.  Thus, $p_i$ would skewly top $B_0$ and $B_1$, a contradiction.  And a similar argument could be made if $\varphi(p_i) \subset \varphi(p_{1-i})$ held.

\medskip
\noindent {\em Case 3: for some $j \in \{0,1\}$, $p_0,p_1 \geq_P \max B_j$.} Since $p_0$ and $p_1$ are distinct, we must have $p_i >_P \max B_j$ for some $i \in \{0,1\}$.  Since $\max B_0$ and $\max B_1$ are $\leq_P$-incomparable, this means that $p_i \in \varphi(p_i)$.  So if $p_i \notin \varphi(p_{1-i})$ then $\varphi(p_i) \neq \varphi(p_{1-i})$.  And if $p_i \in \varphi(p_{1-i})$ then $p_i <_P p_{1-i}$ and hence $p_{1-i} \in \varphi(p_{1-i}) - \varphi(p_i)$, so again $\varphi(p_i) \neq \varphi(p_{1-i})$.  Now suppose $p_{1-i} <_P p_{i}$ for some $i$, so that $\pi(p_{1-i}) \subset \pi(p_i)$.  Then as $\varphi(p_0) = \pi(p_0) \cup \{q^*\} - \{\max B_0, \max B_1 \}$ and $\varphi(p_1) = \pi(p_1) \cup \{q^*\} - \{\max B_0, \max B_1 \}$, we have $\varphi(p_{1-i}) \subseteq \varphi(p_i)$ and hence $\varphi(p_{1-i}) \subset \varphi(p_i)$ since $\varphi(p_{1-i}) \neq \varphi(p_i)$.  Conversely, suppose $\varphi(p_{1-i})  \subset \varphi(p_{i})$.  The only way it could fail to be the case that $\pi(p_{1-i}) \subset \pi(p_i)$ is if $\max B_{1-j} \notin \pi(p_{1-i})$.  But every $p <_P \max B_{1-j}$ belongs to $\varphi(p_{1-i}) - \{q^*\}$ and hence to $\varphi(p_i) - \{q^*\} \subseteq \pi(p_i)$, meaning $p \leq_P p_i$, so if this were the case then $p_i$ would skewly top $B_0$ and $B_1$.  It must thus be that $\pi(p_{1-i}) \subset \pi(p_i)$ and hence that $p_{1-i} <_P p_i$, as desired.

\medskip
Our next claim is that $\varphi$ is parsimonious.  Fixing $p \in P$, we first verify condition (1a) of Definition \ref{D:set_infinite}.   If $p \ngeq_P \max B_0, \max B_1$, then there is nothing to show since $\varphi(p) = \pi(p)$.  If $p = \max B_j$ for some $j \in \{0,1\}$, then $\varphi(p) = \bigcup_{p' <_P p} \varphi(p')  \cup \{q^*\}$ and $q^* \notin \varphi(p')$ for any $p' <_P p$ since necessarily $p' \ngeq_P \max B_0,\max B_1$.  If $p >_P \max B_j$ for some $j$, then $\varphi(p) = \bigcup_{p' <_P p} \varphi(p')  \cup \{p\}$.  In this case, $p \notin \varphi(p')$ for any $p' <_P p$ since as $p \neq q^*$ this would mean that $p \in \pi(p')$ and hence that $p \leq_P p' <_P p$.  In any case, then, $\left| \varphi(p) - \bigcup_{p' <_P p} \varphi(p') \right| = 1$.

We now verify condition (1b) of Definition \ref{D:set_infinite}.  Given $q \in \varphi(p)$, we either have that $q = q^*$ and $\max B_j \leq_P p$ for some $j \in \{0,1\}$, or that $q \in P$ and $q \leq_P p$.  If we apply the argument just given to $q$ instead of to $p$  then it follows that in the former case $\{q\} = \varphi(\max B_j) - \bigcup_{p' <_P \max B_j} \varphi(p')$, and that in the latter case $\{q\} = \varphi(q) - \bigcup_{p' <_P q} \varphi(p')$.

Finally, it follows that $\bf P$ is not saturated.  Indeed, as the preceding argument shows, $\alpha_\varphi (\max B_0) = q^* = \alpha_\varphi(\max B_1)$.  Hence, $\alpha_\varphi$ is not injective.

\medskip
\noindent ($\Longleftarrow$) Fix a partial order ${\bf P} = (P, \leq_P)$.  Fix a parsimonious set representation $\varphi: P \to \mathcal{P}(Q)$ and suppose that $\alpha_\varphi$ is not injective, so that $\alpha_\varphi(p_0) = \alpha_\varphi(p_1)$ for some distinct $p_0,p_1 \in P$.  Then by definition of $\alpha_\varphi$, it follows that $p_0$ and $p_1$ are $\leq_P$-incomparable and not minimal in $P$.  For $i = 0,1$, let $I_i$ be the set of all $p <_P p_i$ in $P$ which are $\leq_P$-incomparable with $p_{1-i}$, and let $C_i$ consist of all $p <_P p_i$ in $P$ which are $\leq_P$-comparable with $p_{1-i}$.  Note that necessarily $p <_P p_{1-i}$ for all $p \in C_i$.  This implies that each $I_i$ must be nonempty as otherwise we would have $\varphi(p) \subset \varphi(p_{1-i})$ for all $p <_P p_i$ by virtue of $\varphi$ being a set representation, which would mean that $\varphi(p_i) \subseteq \varphi(p_{i-1})$ and hence that $p_i \leq_P p_{1-i}$.

Thus, $I_0 \cup \{p_0\}$ and $I_1 \cup \{p_1\}$ are parallel bouquets in $\bf P$ with $p_0$ and $p_1$ as their respective maxima.  Now suppose $m \in P$ and $i \in \{0,1\}$ is such that $p_i <_P m$ and $p <_P m$ for all $p <_P p_{1-i}$.  Then $\alpha_\varphi(p_{1-i}) \in \varphi(p_i) \subset \varphi(m)$ and $\varphi(p) \subset \varphi(m)$ for all $p <_P p_{1-i}$ and thus
$$
\varphi(p_{1-i}) = \{\alpha_\varphi(p_i)\} \cup \bigcup_{p <_P p_{1-i}} \varphi(p) \subseteq \varphi(m),
$$
which gives $p_{1-i} \leq m$.  Thus, $I_0 \cup \{p_0\}$ and $I_1 \cup \{p_1\}$ are not skewly topped.
\end{proof}

The theorem shows why the move from fans in the finite case to bouquets in the infinite case was necessary.  For consider the partial order $\bf P$ with domain 
$$
P = \{l_i : i \in \N\}  \cup \{l\} \cup \{r_i: i \in \N\} \cup \{r\} \cup \{t_i: i \in \N\},
$$
and ordering $\leq_P$ defined by (the transitive closure of) the following: for all $i <_\N j$,
\begin{itemize}
\item $l_i <_P l_j <_P l$,
\item $r_i <_P r_j <_P r <_P t_i <_P t_j$,
\item $l_i <_P t_i$.
\end{itemize} (See Figure \ref{fig}.)
\begin{figure}[htp]
\begin{center}
$$
\xygraph{
!{<0cm,0cm>;<1cm,0cm>:<0cm,0.7cm>::}
!{(0,0) }*+{\bullet}="l0"
!{(0,1) }*+{\bullet}="l1"
!{(0,2)}*+{\vdots }="ld"
!{(0,3) }*+{\bullet}="ss"
!{(-1,0) }*+{l_0}
!{(-1,1) }*+{l_1}
!{(-1,3) }*+{l}
!{(1,0) }*+{\bullet}="r0"
!{(1,1)}*+{\bullet}="r1"
!{(1,2)}*+{\vdots}="rd"
!{(1,3) }*+{\bullet}="ts"
!{(1,4) }*+{\bullet}="t0"
!{(1,5) }*+{\bullet}="t1"
!{(1,6)}*+{\vdots }
!{(2,0) }*+{r_0}
!{(2,1)}*+{r_1}
!{(2,3) }*+{r}
!{(2,4) }*+{t_0}
!{(2,5) }*+{t_1}
"r0"-"r1"
"l0"-"l1"
"l1"-"ld"
"ld"-"ss"
"r1"-"rd"
"rd"-"ts"
"t0"-"t1"
"ts"-"t0"
"t0"-"l0"
"t1"-"l1"
} 
$$
\caption{\,}\label{fig}
\end{center}
\end{figure}
Then if $F_0$ and $F_1$ are parallel fans in $\bf P$, it must be that $|F_0 | = |F_1| = 2$, and that one of the two fans, say $F_0$, only contains elements $\leq_P$-incomparable with $r$, while the other only contains elements $\leq_P$-incomparable with $l$.  Thus either $F_0 = \{l_i, l_j\}$ for some $i <_\N j$, or $F_0 = \{l_i,l\}$ for some $i$.  In either case, $F_1$ must consist of some elements $<_P t_i$, and $t_i$ must consequently skewly top $F_0$ and $F_1$.  On the other hand, $B_0 = \{l_0,l_1,\ldots\} \cup \{l\}$ and $B_1 = \{r_0,r_1,\ldots\} \cup \{r\}$ are parallel bouquets in $\bf P$ which are clearly not skewly topped by any element of $P$.  By the theorem, $\bf P$ is not saturated.

\section{Reverse mathematics}\label{S:RM}

Reverse mathematics is an area of mathematical logic devoted to classifying mathematical theorems according to their proof theoretic strength.  The goal is to calibrate this strength according to how much comprehension is needed to establish the existence of the sets needed to prove the theorem (i.e., according to how complicated the formulas specifying such sets must be allowed to be).  This is a two-step process. The first involves searching for some weak comprehension scheme sufficient to prove the theorem, while the second gives sharpness by showing that the theorem is in fact equivalent to this comprehension scheme.

In practice, we use for these comprehension schemes certain subsystems of second order arithmetic.  As our base theory we use a weak subsystem called $\RCA$ which roughly corresponds to computable or constructive mathematics.  A strictly stronger system is $\WKL$, obtained by adding to the axioms of $\RCA$ the comprehension scheme asserting that every infinite binary tree has an infinite branch, and stronger still is $\ACA$, which adds comprehension for sets described by arithmetical formulas (i.e., formulas whose quantifiers range over only number variables).  Many theorems are known to be either provable in $\RCA$ or else equivalent over $\RCA$ to one of $\WKL$ or $\ACA$; see \cite{Si}, Chapter 1 for a partial list of examples, and for an overview of other subsystems of second order arithmetic.

We turn to analyzing the proof theoretic strength of Theorem \ref{T:main}, assuming familiarity with the subsystems mentioned above.  For interval orders, the equivalences between various set-theoretic and algebraic characterizations were studied in this context by Marcone \cite{Ma}.  For example, it turns out that Theorem \ref{T:intchar} is provable in $\RCA$ (\cite{Ma}, Theorems 2.13 and 4.2), but that other characterizations of interval orders are harder to prove:

\begin{thm}[Marcone \cite{Ma}, Theorem 5.6]
Over $\RCA$, the following are equivalent:
\begin{enumerate}
\item $\WKL$;
\item a partial order is an interval order if and only if it admits an injective interval representation.
\end{enumerate}
\end{thm}

For our purposes, we begin by formalizing the concept of set representation in the language of second order arithmetic.

\begin{defn}The following definitions are made in $\RCA$. Let ${\bf P} = (P, \leq_P)$ be a partial order.  A {\em set representation} of $\bf P$ is a subset $\varphi$ of $P \times Q$ for some set $Q$ such that if we abbreviate $\{q \in Q : \langle p, q \rangle \in \varphi \}$ by $\varphi(p)$ then for all $p,p' \in P$
\begin{enumerate}
\item $p \neq p' \implies \varphi(p) \neq \varphi(p')$,
\item and $p <_P p' \iff \varphi(p) \subset \varphi(p')$.
\end{enumerate}
\end{defn}

Parsimony is then formalized in a straightforward way, along with all the combinatorial notions from Definitions \ref{D:fan} and \ref{D:bouquet}.  Formalizing saturation, on the other hand, presents us with two options (we deliberately use the same term for both):

\begin{defn}\label{D:sat_rca}
The following definitions are made in $\RCA$.  Let ${\bf P} = (P, \leq_P)$ be a partial order.
\begin{enumerate}
\item $\bf P$ is {\em saturated} if for every parsimonious set representation $\varphi \subseteq P \times Q$ of $\bf P$ it holds that for all $p_0,p_1 \in P$ and all $q_0,q_1 \in Q$, if $p_0 \neq p_1$ and $q_i = \varphi(p_i) - \bigcup_{p' <_P p_i} \varphi(p')$ for each $i \in \{0,1\}$, then $q_0 \neq q_1$.
\item $\bf P$ is {\em saturated} if for every parsimonious set representation $\varphi \subseteq P \times Q$ of $\bf P$, the map $\alpha_\varphi : P \to Q$ exists and is injective.
\end{enumerate}
\end{defn}

In ordinary terms the two definitions are, of course, one and the same.  But in the present context they need not be because the existence of the map $\alpha_\varphi$ may not always be provable in $\RCA$.  The following pair of propositions show that this can indeed happen.  Thus, while formulating saturation according to Definition \ref{D:sat_rca}~(2) may be more natural, the set theoretic assumptions necessary to carry out the proof of Theorem \ref{T:main} become much higher.

\begin{prop}\label{P:weak}
It is provable in $\RCA$ that a partial order is saturated according to Definition \ref{D:sat_rca}~(1) if and only if every two parallel bouquets in it are skewly topped.
\end{prop}

\begin{proof}
$\RCA$ suffices to carry out the left-to-right direction of the proof of Theorem \ref{T:main}.  For the right-to-left direction, fix a partial order ${\bf P} = (P, \leq_P)$ and a parsimonious set representation $\varphi \subseteq P \times Q$.  Suppose there exists $p_0 \neq p_1$ in $P$ such that $\varphi(p_0) - \bigcup_{p' <_P p_0} \varphi(p') = \varphi(p_1) - \bigcup_{p' <_P p_1} \varphi(p')$.  Then we can argue as in the right-to-left direction of the proof of Theorem \ref{T:main} that there exist parallel bouquets in $\bf P$ which are not skewly topped.
\end{proof}

\begin{prop}
Over $\RCA$, the following are equivalent:
\begin{enumerate}
\item $\ACA$;
\item for every parsimonious set representation $\varphi$ of a partial order, the map $\alpha_\varphi$ exists;
\item a partial order is saturated according to Definition \ref{D:sat_rca}~(2) if and only if every two parallel bouquets in it are skewly topped;
\item a partial order is saturated according to Definition \ref{D:sat_rca}~(1) if and only if it is saturated according to Definition \ref{D:sat_rca}~(2).
\end{enumerate}
\end{prop}

\begin{proof}
For every parsimonious set representation $\varphi$ of a partial order $(P, \leq_P)$ we have $\alpha_\varphi$
% = \{ \langle p, q \rangle : q \in \varphi(p) \wedge (\forall p' \in P)[\,p' <_P p \to q \notin \varphi(p')\,]\}$,
arithmetically definable, so (1) implies (2).  By Proposition \ref{P:weak} it follows that (2) implies (3), and obviously the equivalence of (1) and (3) implies the equivalence of (1) and (4).

It thus remains only to show that (3) implies (1).  To this end, we prove from (3) that the range of every injective function $f : \N \to \N$ exists (this is equivalent; see \cite{Si}, Theorem III.1.3). So fix $f$ and define a partial order ${\bf P} = (P, \leq_P)$ as follows.  Let $P = \{p_{i,s} : i,s \in \N\}$.
\begin{itemize}
\item For all $i <_\N j$, let $p_{i,s} >_P p_{j,t}$ for all $s,t \in \N$.
\item For each $i$ and all $s <_\N t$, let $p_{i,s} <_P p_{i,t}$ if $s > 0$ and $f(t-2) = i$, and let $p_{i,s} >_P p_{i,t}$ otherwise.
\end{itemize}
In other words, if $f(t) \neq i$ for all $t$ then we have $p_{i,s} >_P p_{i,t}$ for all $s <_\N t$, while if $f(t) = i$ for some $t$ then we have $p_{i,0} >_P  p_{i,t+2} >_P p_{i,s} >_P p_{i,s'}$ for all $s <_\N s'$ in $\N - \{0,t+2\}$.  $\RCA$ suffices to show that $\bf P$ exists, that it is a linear order, and that every element has an immediate $\leq_P$-predecessor.  In particular, linearity implies that there are no parallel bouquets in $\bf P$, so by (3) $\bf P$ must be saturated according to Definition \ref{D:sat_rca}~(2).

Define
$$
\varphi = \{ \langle p,p' \rangle \in P \times P: p >_P p' \wedge (\forall i \in \N)[\, p' \neq p_{i,0}\, ]\},
$$
 which exists by $\Sigma^0_0$ comprehension and is clearly a set representation of $\bf P$.  If we let $p^-$ denote the immediate $\leq_P$-predecessor of each $p \in P$, then we see that $\{ p^-\} = \varphi(p) - \bigcup_{p' <_P p} \varphi(p')$.  Furthermore, if $q \in \varphi(p)$ for some $p = p_{i,s}$ then $p >_P q$ and $q = p_{j,t}$ for some $j \geq_\N i$ and $t >_\N 0$, so $q = p_{j,t'}^-$ for some $p_{j,t'} \leq_P p$.  Thus, $\varphi$ is parsimonious.
 
 It follows that $\alpha_\varphi : P \to P$ exists and is injective, and by the preceding discussion we have $\alpha_\varphi(p) = p^-$ for all $p$.  Let $R = \{i \in \N: \alpha_\varphi(p_{i,0}) \neq p_{i,1} \}$, which exists by $\Sigma^0_0$ comprehension.  Then by construction of $\leq_P$, we have that $i \in R$ if and only if $p_{i,0}^- = p_{i,t+2}$ for some $t$ such that $f(t) = i$, which in turn holds if and only if $i \in \ran f$.  Hence, the range of $f$ is equal to $R$ and so consequently exists.  This completes the proof.
\end{proof}

\bibliographystyle{plain}
\bibliography{setrepbib.bib}

\end{document}